\documentclass[preprint,12pt]{elsarticle}
\usepackage{amsthm,amsmath,amssymb}
\usepackage[colorlinks=true,citecolor=black,linkcolor=black,urlcolor=blue]{hyperref}
\usepackage{graphicx}
\usepackage{float}
\usepackage{epstopdf}
\usepackage{epsfig}
\usepackage{extarrows,chngpage,array,float,eqnarray}     

\theoremstyle{plain}
\newtheorem{theorem}{Theorem}
\newtheorem{lemma}[theorem]{Lemma}
\newtheorem{corollary}[theorem]{Corollary}

\newtheorem{observation}[theorem]{Observation}

\theoremstyle{definition}
\newtheorem{definition}[theorem]{Definition}

\theoremstyle{remark}
\newtheorem{remark}[theorem]{Remark}

\title{\textbf{On coefficients of the interior and exterior polynomials}}
\author{
	\small Xiaxia Guan, \ \ Xian'an Jin\footnote{Corresponding author}\\[0.2cm]
	\small School of Mathematical Sciences, Xiamen University,\\
	\small Xiamen, Fujian 361005, China\\[0.2cm]
	\small E-mails: gxx0544@126.com; xajin@xmu.edu.cn}
\date{} 
\begin{document}
\begin{abstract}
The interior polynomial and the exterior polynomial are generalizations of valuations on $(1/\xi,1)$ and $(1,1/\eta)$ of the Tutte polynomial $T_G(x,y)$ of graphs to hypergraphs, respectively. The pair of hypergraphs induced by a connected bipartite graph are abstract duals and are proved to have the same interior polynomial, but may have different exterior polynomials. The top of the HOMFLY polynomial of a special alternating link coincides with the interior polynomial of the pair of hypergraphs induced by the Seifert graph of the link. Let $G=(V\cup E, \varepsilon)$ be a connected bipartite graph. In this paper, we mainly study the coefficients of the interior and exterior polynomials. We prove that the interior polynomial of a connected bipartite graph is interpolating. We strengthen the known result on the degree of the interior polynomial for connected bipartite graphs with 2-vertex cuts in $V$ or $E$. We prove that interior polynomials for a family of balanced bipartite graphs are monic and the interior polynomial of any connected bipartite graph can be written as a linear combination of interior polynomials of connected balanced bipartite graphs. The exterior polynomial of a hypergraph is also proved to be interpolating. It is known that the coefficient of the linear term of the interior polynomial is the nullity of the bipartite graph, we obtain a `dual' result on the coefficient of the linear term of the exterior polynomial: if $G-e$ is connected for each $e\in E$, then the coefficient of the linear term of the exterior polynomial is $|V|-1$. Interior and exterior polynomials for some families of bipartite graphs are computed.
\vskip0.2cm

\noindent{\bf Keywords:}\ Hypergraph; Bipartite graph; Interior polynomial; Exterior polynomial; Coefficient

\noindent {\bf MSC(2020)}\ 05C31; 05C65

\end{abstract}
\maketitle

\section{Introduction}
\noindent

The Tutte polynomial $T_{G}(x, y)$ \cite{Tutte} is an important and well-studied branch of graph and matroid theory, having wide applications from knot theory to statistical mechanics. Motivated by the study of the HOMFLY polynomial \cite{Freyd,Hom2}, which is a generalization of the celebrated Jones polynomial \cite{Jones} in knot theory, of special alternating links,  K\'{a}lm\'{a}n introduced the interior polynomial $I_{\mathcal{H}}(x)$ and the exterior polynomial $X_{\mathcal{H}}(y)$ \cite{Kalman1} (see Definitions \ref{def 2} and \ref{def 3}) via interior and exterior activities of hypertrees (firstly introduced as the left or right degree vector in \cite{Post}). The two polynomials generalized valuations of $T_G(1/x,1)$ and $T_G(1,1/y)$ of the Tutte polynomials to hypergraphs, respectively.

K\'{a}lm\'{a}n and Postnikov proved that the interior polynomial and the exterior polynomial are two invariants of hypergraphs by a straightforward argument and indirect approach by counting Ehrhart-type lattice points in \cite{Kalman1} and \cite{Kalman3}. Moreover, in \cite{Kalman3} they also showed that the interior polynomial is a Tutte-type invariant of bipartite graphs via the Ehrhart polynomial of the root polytope $Q$ and the $h$-vector of any triangulation of $Q$, in other words, the pair of hypergraphs induced by a bipartite graph have the same interior polynomials.
In \cite{Kalman2,Kalman3}, the authors established a relation between the top of the HOMFLY polynomial of any special alternating link and the interior polynomial of the Seifert graph (which is a bipartite graph) of the link.
Recently, Kato \cite{Keiju} introduced the signed interior polynomial of signed bipartite graphs and extended the relation to any oriented links, and in \cite{Kalman4} K\'{a}lm\'{a}n and T\'{o}thm\'{e}r\'{e}sz defined two one-variable generating functions $\widetilde{I}$ and $\widetilde{X}$ using two `embedding activities' and proved that the generating function of internal embedding activities coincides with the interior polynomial.

In this paper, we mainly study the coefficients of the interior and exterior polynomials of hypergraphs. We prove that the interior polynomial and the exterior polynomial of connected bipartite graphs are both interpolating, that is, if $I_{G}(x)=\sum\limits_{i=0}^{m}a_{i}x^{i}$ and $X_{G}(y)=\sum\limits_{j=0}^{n}b_{j}y^{j}$ of  a connected bipartite graph $G$, then $a_{i}\neq 0$ and $b_{j}\neq 0$ for any $i=0,1,\cdots,n$ and $j=0,1,\cdots,m$.

K\'{a}lm\'{a}n \cite{Kalman1} showed that the degree of interior polynomial of $G$ is at most $\min\{|V|,|E|\}-1$ for a bipartite graphs $G=(V\cup E, \varepsilon)$, and the degree of exterior polynomial of a hypergraph $\mathcal{H}=(V,E)$ is at most $|E|-1$. In this paper we study the degree of the interior polynomial of connected bipartite graphs with 2-vertex cuts in $E$ and prove that for a connected bipartite graph $G=(V\cup E, \varepsilon)$ wtih a 2-vertex cut $\{e_{1}, e_{2}\}$ in $E$, then the degree of the interior polynomial of $G$ is at most $\min\{|E|-1$, $|V|-t+1$\}, where $t$ is the number of connected components of $G-\{e_{1}, e_{2}\}$.

We show that interior polynomials for a family of balanced bipartite graphs are monic. It is well-known that the Tutte polynomial of graphs has a deletion-contraction formula. Using deletion-contraction relations of the interior polynomials repeatedly, we show that the interior polynomial of any connected bipartite graph can be represented as a linear combination of the interior polynomials of connected balanced bipartite graphs.

K\'{a}lm\'{a}n \cite{Kalman1} proved that both the interior polynomial and the exterior polynomial have constant term 1 for any connected bipartite graph. Moveover, they proved that the coefficient of the linear term in the interior polynomial for any connected bipartite graph is the nullity of the bipartite graph. In this paper, we prove that if $G-e$ is connected for each $e\in E$, then the coefficient of the linear term of the exterior polynomial is $|V|-1$.

As far as we know that the interior polynomial and exterior polynomial of only few bipartite graphs are known. In this paper, we also compute the interior polynomial and exterior polynomial for some families of bipartite graphs such as trees, even cycles, complete bipartite graphs and so on.

The paper is organized as follows. In Section 2 we give preliminaries on the two polynomials and make some necessary preparations. In Section 3, we study interpolatory property, the degree, the monic property and the representation of the interior polynomial. Interpolatory property and the coefficient of the linear term of the exterior polynomial are studied in Section 4. In Section 5, the interior polynomial and exterior polynomial for some families of bipartite graphs are computed.

\section{Preliminaries}
\noindent

In this section, we will give some definitions and summarise some known results.

\begin{definition}
A \emph{hypergraph} is a  pair $\mathcal{H}=(V,E)$, where $V$ is a finite set and $E$ is a finite multiset of non-empty subsets of $V$. Elements of $V$ are called \emph{vertices} and  elements of $E$ are called \emph{hyperedges}, respectively, of the hypergraph.
\end{definition}
A hypergraph is a generalization of a graph with loops and multiple edges allowed. For a hypergraph $\mathcal{H}$, its associated bipartite graph $Bip \mathcal{H}$ is defined as follows.
\begin{definition}
The sets $V$ and $E$ are the colour classes of the bipartite graph $Bip \mathcal{H}$, and an element $v$ of $V$ is connected to an element $e$ of $E$ in $Bip \mathcal{H}$ if and only if $v\in e$.
\end{definition}

Clearly, $Bip \mathcal{H}$ is simple, that is, it has no multiple edges in $Bip \mathcal{H}$ (and clearly, it has no loops). If $\mathcal{H}$ is a graph, then $Bip \mathcal{H}$ is the subdivision of $\mathcal{H}$, but two multiple edges obtained by subdividing a loop should be replaced by a single edge.

Conversely, a pair of hypergraphs $\mathcal{H}=(V,E)$ and $\overline{\mathcal{H}}=(E,V)$ can be recovered from a bipartite graph (without multiple edges) if we specify $V$ and $E$ as vertex set of the hypergraph, respectively. These two hypergraphs $\mathcal{H}=(V,E)$ and $\overline{\mathcal{H}}=(E,V)$ are called \emph{abstract duals}.

We say a hypergraph $\mathcal{H}=(V,E)$ is \emph{connected} if $Bip \mathcal{H}$ is connected. Throughout the paper we consider connected hypergraphs.

\begin{definition} \label{def 1}
Let $\mathcal{H}=(V,E)$ be a connected hypergraph. A hypertree in $\mathcal{H}$ is a function $f$: $E\rightarrow N=\{0,1,2,\cdots\}$ so that a spanning tree $\tau$ of its associated bipartite graph $Bip \mathcal{H}$ can be found with the degree of $e$ in $\tau$, $d_{\tau}(e)=f(e)+1$ for each $e\in E$. We call that $\tau$ \emph{realises} or \emph{induces} $f$. We denote the set of all hypertrees in $\mathcal{H}$ with $B_{\mathcal{H}}$.
\end{definition}

Hypertrees generalize spanning trees of graphs in the sense that an edge $e$ is in the tree if and only if $f(e)=1$ and not in the tree if and only if $f(e)=0$. As a generalization of the rank of a graph, the following parameter $\mu(E')$ is introduced for each $E'\subset E$.

\begin{definition}
Let $G=(V\cup E, \varepsilon)$ be a connected bipartite graph. For a subset $E'\subset E$, let $G|_{E'}$ denote the bipartite graph formed by $E'$, all edges of $G$ incident with elements of $E'$ and their endpoints in $V$. We denote $\mu(E')=0$ for $E'=\emptyset$, and $\mu(E')=|\bigcup E'|-c(E')$ for $E'\neq \emptyset$, where $\bigcup E'=V\cap (G|_{E'})$ and $c(E')$ is the number of connected components of $G|_{E'}$.
\end{definition}

The following facts  hold  for  hypertrees of a hypergraph.

\begin{theorem} [\cite{Kalman1}] \label{thm 1}
Let $\mathcal{H}=(V,E)$ be a connected hypergraph and $Bip \mathcal{H}$ be the bipartite graph associated to the hypergraph $\mathcal{H}$. Let $f$ be a hypertree of $\mathcal{H}$. Then
\begin{enumerate}
\item[(1)] $0\leq f(e)\leq d_{Bip \mathcal{H}}(e)-1$ for all $e\in E$;

\item[(2)] $\sum\limits_{e\in E}f(e)=|V|-1$;

\item[(3)] $\sum\limits_{e\in E'}f(e)\leq \mu(E')$ for all $ E'\subset E$.
\end{enumerate}
\end{theorem}

\begin{definition}
Let $\mathcal{H}=(V,E)$ be a connected hypergraph. Let $f$ be a hypertree of $\mathcal{H}$.  For the subset $E'\subset E$, we say that $E'$ is  tight at $f$, if  $\sum\limits_{e\in E'}f(e)= \mu(E')$ holds.
\end{definition}

It is clear that $\emptyset$ and $E$ are always tight at any hypertree, and if $E'$ is tight at $f$ and $\sum\limits_{e\in E'}f(e)= \sum\limits_{e\in E'}g(e)$ for another hypertree $g$, then $E'$ is also tight at $g$. Moreover, the next Theorem follows immediately from Theorem 44.2 in \cite{Schrijver} since $\mu$ is submodular.

\begin{theorem}\label{thm 2}
Let $\mathcal{H}=(V,E)$ be a connected hypergraph. Let $f$ be a hypertree of $\mathcal{H}$.  If the subsets $A\subset E$ and $B\subset E$ are both tight at $f$, then $A\cap B$ and $A\cup B$ are both tight at $f$.
\end{theorem}

\begin{theorem}[\cite{Kalman1}]\label{thm}
Let $\mathcal{H}=(V,E)$ be a connected hypergraph and $Bip \mathcal{H}$ be the bipartite graph associated to  $\mathcal{H}$. Let $f$ be a hypertree of $\mathcal{H}$, and $\tau$ be a spanning tree of $Bip \mathcal{H}$ inducing $f$.  The subset $A\subset E$ is tight at $f$ if and only if $\tau|_{A}$ is a spanning forest of $Bip \mathcal{H}|_{A}$, i.e., components of $\tau|_{A}$ are exactly spanning trees of components of $Bip \mathcal{H}|_{A}$.
\end{theorem}

\begin{definition}
Let $\mathcal{H}$ be a connected hypergraph. Let $f$ be a hypertree and $e$, $e'$ be two hyperedges of $\mathcal{H}$. We say $f$ is the hypertree  so that \emph{a transfer of valence} is possible from  $e$ to $e'$ if  the function $f'$ obtained from $f$ by decreasing $f(e)$ by 1 and increasing $f(e')$ by 1 is also a hypertree. We also say that $f$ and $f'$ are related by a transfer of valence from  $e$ to $e'$.
\end{definition}

Now we introduce (internal and external) activity and inactivity of hypertrees of a hypergraph with a fixed totally ordering on $E$.

\begin{definition} \label{def 2}
Let $\mathcal{H}=(V,E)$ be a connected hypergraph and $f$ be a hypertree. A hyperedge $e\in E$ is internally active with respect to the hypertree $f$  if one cannot decrease $f(e)$ by 1  and increase $f(e')$  of a hyperedge $e'$ smaller than $e$ by 1 so that another hypertree results. We say that a hyperedge $e\in E$ is internally inactive with respect to the hypertree $f$  if it is not internally active. Let $\iota(f)$ and $\overline{\iota}(f)$ denote the number of internally active  hyperedges and internally inactive hyperedges, respectively, with respect to $f$. These two values are called the internal activity and internal inactivity, respectively, of $f$.
\end{definition}

Similar to internal activity and inactivity, there are external activity and inactivity of a hypertree in a hypergraph with a fixed totally ordering on $E$.

\begin{definition}\label{def 3}
Let $\mathcal{H}=(V,E)$ be a connected hypergraph and $f$ be a hypertree. A hyperedge $e\in E$ is externally active with respect to the hypertree $f$  if one cannot increase $f(e)$ by 1  and decrease $f(e')$  for some hyperedge $e'<e$ by 1 so that another hypertree results. We say that a hyperedge $e\in E$ is externally inactive with respect to $f$  if it is not externally active. Let $\epsilon(f)$ and $\overline{\epsilon}(f)$ denote the number of externally  active  hyperedges and externally  inactive hyperedges, respectively, with respect to $f$. These two values are called the external activity and external inactivity, respectively, of $f$.
\end{definition}

We have that $\iota(f)+\overline{\iota}(f)=\epsilon(f)+\overline{\epsilon}(f)=|E|$ by definitions above. Now, we can define the interior polynomial and the exterior polynomial of hypergraphs as follows.

\begin{definition}
Let $\mathcal{H}=(V,E)$ be a connected hypergraph. For some fixed  order on $E$, we denote the interior polynomial $I_{\mathcal{H}}(x)=\sum\limits_{f\in B_{\mathcal{H}}} x^{\overline{\iota}(f)}$ and the exterior polynomial  $X_{\mathcal{H}}(y)=\sum\limits_{f\in B_{\mathcal{H}}} y^{\overline{\epsilon}(f)}$.
\end{definition}

K\'{a}lm\'{a}n and Postnikov proved that both the interior polynomial and the exterior polynomial are well-defined in \cite{Kalman1} and \cite{Kalman3}, i.e, they do not depend on the order on $E$.

\begin{theorem}[\cite{Kalman1}] \label{thm 3}
Let $\mathcal{H}=(V,E)$ be a connected hypergraph. Then the interior polynomial and the exterior polynomial of $\mathcal{H}$ do not depend on the chosen order on $E$.
\end{theorem}

If $G=(V, E)$ is a graph with the Tutte polynomial $T_{G}(x, y)$, then its (viewed as a hypergraph) interior polynomial is $x^{|V|-1}T(1/x,1)$, and its exterior polynomial is $y^{|E|-|V|+1}T(1,1/y)$. Moreover, the interior polynomial is an invariant of bipartite graphs \cite{Kalman3}.

\begin{theorem}[\cite{Kalman3}] \label{thm 4}
 If $\mathcal{H}$ and $\overline{\mathcal{H}}$ are abstract dual hypergraphs, then $I_{\mathcal{H}}(x)=I_{\overline{\mathcal{H}}}(x)$.
\end{theorem}

However, the two abstract dual hypergraphs induced by a bipartite graph may have different exterior polynomials. When we say the exterior polynomial of a bipartite graph, we shall specify which colour class is the set of hyperedges. In addition, it is not difficult to see that multiple edges in a bipartite graph do not affect the set of hypertrees and so do the two polynomials. Recall that for a hypergraph $\mathcal{H}$ its associated bipartite graph $Bip \mathcal{H}$ is always simple.

The support of a polynomial $f(x)=\sum\limits_{i=0}^{m}a_{i}x^{i}$ is the set $supp(f)=\{i|a_{i}\neq 0\}$ of  indices of the non-zero coefficients.
\begin{definition}
The \emph{degree} of the polynomial $f$ is the maximum of its $supp(f)$. The polynomial $f$ is called \emph{interpolating} if its $supp(f)$ is an integer interval $[n_1, n_2]$
of all integers from $n_1$ to $n_2$, inclusive.
\end{definition}

\section{The interior polynomial}
\noindent

In this section, we study interpolatory property of the interior polynomial, the degree of the interior polynomial for bipartite graphs having 2-vertex cuts in $E$ or $V$, monic property of interior polynomials for balanced bipartite graphs and the representation of the interior polynomial of a bipartite graph as a linear combination of those of some balanced bipartite graphs. Since $I_{\mathcal{H}}(x)=I_{\overline{\mathcal{H}}}(x)$ for abstract dual hypergraphs $\mathcal{H}$ and $\overline{\mathcal{H}}$, without loss of generality we assume that vertices in $E$ of a connected bipartite graph $G=(V\cup E, \varepsilon)$ will be regarded as hyperedges of the hypergraph $\mathcal{H}$.

\subsection{Interpolatory property}
\noindent
In this subsection, we will show that the interior polynomial of any connected bipartite graph is interpolating. We need the following lemmas.

\begin{lemma}[\cite{Kalman1}]\label{lem 1}
Let $G=(V\cup E, \varepsilon)$ be a connected bipartite graph. Let $f$ be a hypertree of $G$. Then for any non-empty subset $E' \subset E$,  if $E'$ is not tight at $f$, then $f$ is a hypertree  so that a transfer of valence is possible from  some element of $E\setminus E'$ to  some element of $E'$.
\end{lemma}

\begin{lemma}[\cite{Kalman1}]\label{lem 2}
Let $G=(V\cup E, \varepsilon)$ be a connected bipartite graph, and $f$ be a hypertree of $G$. Let $e_{i}\in E$ $(i=1,2,3)$. If  $e_{1}$ can transfer valence to $e_{2}$  and $e_{2}$ can transfer valence to $e_{3}$  with respect to  $f$,  then  $e_{1}$ can transfer valence to $e_{3}$  with respect to  $f$.
\end{lemma}

\begin{lemma}\label{lem 20}
Let $G=(V\cup E, \varepsilon)$ be a connected bipartite graph, and let $f$ be a hypertree of $G$. Suppose that $e, e'\in E$ and $e\neq e'$. Then $e$ can transfer valence to $e'$  with respect to $f$ if and only if  $f(e)\neq 0$, and every subset $E'\subset  E$, which contains $e'$ and does not contain $e$, is not tight at $f$.
\end{lemma}
\begin{proof}
The necessity is obvious. For sufficiency, let us take $E_{1}=\{e'\}$. Since $E_{1}$ is not tight at $f$ some element of $E\setminus E_{1}$ can transfer valence to $e'$ for $f$ by Lemma \ref{lem 1}. Let $U_{1}$ be the set consisting of  all elements of $E\setminus E_{1}$ that can transfer valence to $e'$ for $f$. If $e\in U_{1}$, then the conclusion is true. If $e\notin U_{1}$, then we take $E_{2}=U_{1}\cup E_{1}$. Note that $E_{2}$ is not tight at $f$. Then some element of $E\setminus E_{2}$ can transfer  valence to some element of $E_{2}$ for $f$ by Lemma \ref{lem 1}. Let $U_{2}$ be the set consisting of  all elements of $E\setminus E_{2}$ that can transfer  valence to some element of $E_{2}$ for $f$. It is obvious that $E_{1}$ is a proper subset of $E_{2}$. Moreover, all elements of $E_{3}=U_{2}\cup E_{2}$ (except for $e'$) can transfer  valence to $e'$ for $f$ by Lemma \ref{lem 2}, and  $E_{2}$ is a proper subset of $E_{3}$. Continue the above process, we will eventually obtain that $e$ can transfer valence to $e'$ for $f$.
\end{proof}

\begin{lemma}\label{lem 3}
Let $G=(V\cup E, \varepsilon)$ be a connected bipartite graph, and let $f$ be a hypertree of $G$. Given an order on $E$, then the hyperedge $e\in E$ is internally inactive with respect to  $f$  if and only if $f(e)\neq 0$, and there exists a hyperedge $e'<e$  so that every subset $E'\subset  E$, which contains $e'$ and does not contain $e$, is not tight at $f$.
\end{lemma}
\begin{proof}
The hyperedge $e$ is internally inactive with respect to  $f$, if and only if for  some hyperedge $e'<e$, $e'$ can be transferred valence from $e$ with respect to $f$ by the definition.
$e$ can transfer valence to $e'$ with respect to $f$ if and only if $f(e)\neq 0$, and every subset $E'\subset  E$, which contains $e'$ and does not contain $e$, is not tight at $f$ by Lemma \ref{lem 20}.
Thus, the conclusion is true.
\end{proof}

\begin{remark}
Lemmas \ref{lem 20} and \ref{lem 3} can also be obtained by using Theorem \ref{thm 1} (3) directly.
\end{remark}

Let $G=(V\cup E, \varepsilon)$ be a connected bipartite graph, and $e_{1}, e_{2}\in E$ and $e_{1}\neq e_{2}$. Let $f_{1}$ and $f_{2}$ be hypertrees of $G$ with $f_{1}(e_{1})<f_{2}(e_{1})$ and $f_{1}(e)=f_{2}(e)$ for all $e\in E$ and $e\neq e_{1},e_{2}$. The following lemma is known.

\begin{lemma}[\cite{Kalman1}]\label{lem 4}
If $f_{1}$ is a hypertree such that valence can  be transferred from $e_{2}$ to $e$, then $f_{2}$ is a hypertree such that valence can  be transferred from $e_{1}$ to $e$.
\end{lemma}

\begin{lemma}\label{lem 5}
Given an order on $E$, and suppose that the hyperedge $e>e_{1}$. If $e$ is internally active with respect to $f_{1}$, then it is also internally active with respect to $f_{2}$.
\end{lemma}
\begin{proof}
If the hyperedge $e$  is internally active with respect to  $f_{1}$, then (i) $f_{1}(e)=0$ or (ii) there exists a subset $U'\subset  E$, which contains $e'$ and does not contain $e$, is tight at $f_{1}$ for any  hyperedge $e'<e$ by Lemma \ref{lem 3}. If (i) holds, $e$ can not be $e_2$, then $f_{2}(e)=f_{1}(e)=0$, clearly, $e$ is internally active with respect to $f_{2}$. If (ii) is true, then in particular, there exists a subset $U_{1}\subset E$, which contains $e_{1}$ and does not contain $e$, is tight at $f_{1}$ since $e_{1}<e$. Take $U=\bigcup\limits_{e'<e} U'$. Then the subset $U\subset  E$, which contains any $e'<e$ and does not contain $e$, is tight at $f_{1}$ by Theorem \ref{thm 2}. Since $e_1\in U$ we have $\sum\limits_{g\in U}f_{1}(g)\leq \sum\limits_{g\in U}f_{2}(g)$. Thus $U$ is also tight at $f_{2}$. This implies that $e$ cannot transfer valence to any $e'<e$ in $f_{2}$, that is, $e$  is internally active with respect to $f_{2}$.
\end{proof}

\begin{lemma} \label{lem 6}
Let $G=(V\cup E, \varepsilon)$ be a connected bipartite graph. If there exists a hypertree with internal inactivity $k$ in $G$, then there exists a hypertree with internal inactivity $k-1$ in $G$ for any integer $k\geq 1$.
\end{lemma}

\begin{proof}
Given an order on $E$, assume that $f_{1}$ is a hypertree with internal inactivity $k$ in $G$, and $e_{m}$ is the smallest internally inactive hyperedge  with respect to $f_{1}$. Let $f$ be a hypertree with internal inactivity $k$ and the hyperedge $e_{m}$ is the smallest internally inactive hyperedge  with respect to $f$ so that the entry $f(e_{m})$ of $e_{m}$ is smallest. Let $e_{n}$ be the smallest hyperedge that can be transferred valence from $e_{m}$ with respect to $f$. Then there is the hypertree $g$ with $g(e_{m})=f(e_{m})-1$, $g(e_{n})=f(e_{n})+1$ and $g(e)=f(e)$ for all $e\neq e_{m},e_{n}$.  Next we prove that the hyperedge $e$ is internally active with respect to $g$ if and only if $e$ is internally active with respect to $f$ for every hyperedge $e\neq e_{m}$ and will show that $g$ is a hypertree with internal inactivity $k-1$.

For any hyperedge $e$ with $e<e_{n}$, since $e$ is internally active with respect to $f$, we have (i) $f(e)=0$ or (ii) there exists $E_{1}\subset  E$, which contains $e'$ and does not contain $e$,  is  tight at $f$  for each hyperedge $e'<e$ by Lemma \ref{lem 3}. If (i) holds, then $g(e)=f(e)=0$, clearly, $e$  is internally active with respect to  $g$. Assume that (ii) is true. Note that $e'<e<e_{n}<e_{m}$. We have that $e_{m}$  can not transfer valence to $e'$ with respect to $f$. Then there exists $E_{2}\subset E$, which contains $e'$ and does not contain $e_{m}$,  is  tight at $f$ by Lemma \ref{lem 20}. Take $E'=E_{1}\cap E_{2}$.  Then $E'$, which contains $e'$ and does not contain $e_{m}$ and $e$,  is  tight at $f$ by Theorem \ref{thm 2}. Note that $\sum\limits_{e\in E'}f(e)\leq \sum\limits_{e\in E'}g(e)$ since $e_m\notin E'$. We have that  $E'$ is  tight at $g$, that is,  there exists $E'\subset E$, which contains $e'$ and does not contain $e$,  is  tight at $g$  for each hyperedge $e'<e$. Thus, by Lemma \ref{lem 3}, $e$ is internally active with respect to $g$.

For the hyperedge $e_{n}$, we claim $e_{n}$ is internally active with respect to $g$. Assume that the opposite is true. Then there is a hyperedge $e'<e_{n}$ that can be transferred valence from $e_{n}$ with respect to $g$. Then there exists the hypertree $f'$ satisfying $f'(e_{m})=f(e_{m})-1$ and $f'(e')=f(e')+1$ and $f'(e)=f(e)$ for all $e\in E$ and $e\neq e_{m},e'$, that is, $e_{m}$ can  transfer valence to $e'$ with respect to $f$. This contradicts  the choice of the hyperedge $e_{n}$.

For the hyperedge $e$ with $e>e_{n}$ that is internally active with respect to $f$ (in fact, $e$ is internally active  with respect to $f$ if $e_{n}<e<e_{m}$), $e$ is also internally active with respect to $g$ by Lemma \ref{lem 5}.

For the hyperedge $e>e_{m}$ that is internally inactive with respect to $f$, we claim that $e$ is internally inactive with respect to $g$. Otherwise, $e$ is internally active with respect to $f$ by Lemma \ref{lem 5}, a contradiction.

It is obvious that $e_{m}$ is internally active with respect to $g$ by the choice of the hypertree $f$. In fact, if $e_{m}$ is internally inactive with respect to $g$, then $g$ will be a hypertree with internal inactivity $k$ and the hyperedge $e_{m}$ is the smallest internally inactive hyperedge with respect to $g$.  Since $f(e_{m})>g(e_{m})$, it contradicts the choice of the hypertree $f$. Thus, $g$ is a hypertree with internal inactivity $k-1$.
\end{proof}
\begin{remark}
In the proof of Lemma \ref{lem 6}, the claim that $e_{n}$ is internally active with respect to $g$ can be obtained by using Lemma \ref{lem 4}.
\end{remark}
\begin{theorem}
The interior polynomial of any connected bipartite graph is interpolating.
\end{theorem}
\begin{proof}
It follows directly from Lemma \ref{lem 6}.
\end{proof}

\subsection{Degree}
\noindent

In this subsection, we will strengthen the result on the degree of the interior polynomial. It has been considered by K\'{a}lm\'{a}n in \cite{Kalman1}, and the author obtained the following result.

\begin{theorem}[\cite{Kalman1}]
Let $G=(V\cup E, \varepsilon)$ be a connected bipartite graph. Then the degree of the interior polynomial of $G$ is at most $\min\{|E|-1$, $|V|-1$\}.
\end{theorem}

The following properties hold for the interior polynomial and the exterior polynomial of bipartite graphs.

\begin{theorem}[\cite{Kalman1}]\label{thm 5}
Let $G_{1}=(V_1\cup E_1, \varepsilon_1)$ and $G_{2}=(V_2\cup E_2, \varepsilon_2)$ be two connected disjoint bipartite graphs.
\begin{enumerate}
\item[(1)] Let $G$ be the connected bipartite graph obtained by identifying a vertex $v_1\in V_1$ and a vertex $v_2\in V_2$ or identifying a vertex $e_1\in E_1$ and a vertex $e_2\in E_2$. Then $I_{G}(x)=I_{G_{1}}(x)I_{G_{2}}(x)$ and $X_{G}(y)=X_{G_{1}}(y)X_{G_{2}}(y)$.
\item[(2)] Let $G$ be the connected bipartite graph obtained by identifying one edge $(v_1,e_1)$ of $G_1$ and one edge $(v_2,e_2)$ of $G_2$, where $v_1\in V_1$, $v_2\in V_2$, $e_1\in E_1$, $e_2\in E_2$, $v_1$ and $v_2$ are identified, and $e_1$ and $e_2$ are identified. Then $I_{G}(x)=I_{G_{1}}(x)I_{G_{2}}(x)$ and $X_{G}(y)=X_{G_{1}}(y)X_{G_{2}}(y)$.
\end{enumerate}
\end{theorem}

Now we consider the case that a connected bipartite graph $G=(V\cup E, \varepsilon)$ has a 2-vertex cut in $E$.

\begin{theorem}\label{thm 6}
Let $G=(V\cup E, \varepsilon)$ be a connected bipartite graph with a 2-vertex cut $\{e_{1}, e_{2}\}$ in $E$. Then the degree of the interior polynomial of $G$ is at most $\min\{|E|-1$, $|V|-t+1$\}, where $t$ is the number of connected components of $G-\{e_{1}, e_{2}\}$.
\end{theorem}
\begin{proof}
We consider vertices on $E$ as  hyperedges of hypergraph $\mathcal{H}$. Given an order on $E$, let $f$ be a hypertree and $\tau$ be a spanning tree inducing $f$ in $G$. Since $G-\{e_{1}, e_{2}\}$ has $t$ connected components $G_{1},G_{2},\cdots,G_{t}$, there is at least one edge of $\tau$ joining each $G_{i}$ $(i=1,2,\cdots,t)$ to one of the hyperedges $e_{1}$ and $e_{2}$, and there is at least a connected component $G_{i}$ $(i=1,2,\cdots,t)$ incident with both $e_{1}$ and $e_{2}$ in the spanning tree $\tau$. Hence $d_{\tau}(e_{1})+d_{\tau}(e_{2})\geq t+1$. Since $d_{\tau}(e)=f(e)+1$ for all $e\in E$, $f(e_{1})+f(e_{2})\geq t-1$. There are two cases.

\noindent{\bf Case 1.} $f(e_{1})+f(e_{2})\geq t$.

In this case, recall that $\sum\limits_{e\in E}f(e)=|V|-1$. We have that $\sum\limits_{e\in E\backslash \{e_{1}, e_{2}\}}f(e)\leq |V|-t-1$. Then the cardinality of the set $\{e\in E\backslash \{e_{1}, e_{2}\}|f(e)\neq 0\}$  is at most $|V|-t-1$. Note that if $e$ is internally inactive with respect to $f$, then $f(e)\neq0$. Thus, there are at most $|V|-t-1$ internally inactive hyperedges in $E\backslash \{e_{1}, e_{2}\}$ with respect to $f$ and there are at most $|V|-t+1$ internally inactive hyperedges in $E$ with respect to $f$.

\noindent{\bf Case 2.} $f(e_{1})+f(e_{2})= t-1$.

In this case we have that $\sum\limits_{e\in E\backslash \{e_{1}, e_{2}\}}f(e)= |V|-t$. Then the cardinality of the set $\{e\in E\backslash \{e_{1}, e_{2}\}|f(e)\neq 0\}$  is at most $|V|-t$. Similarly there are at most $|V|-t$ internally inactive hyperedges in $E\backslash \{e_{1}, e_{2}\}$ with respect to $f$. Without loss of generality, we assume that $e_{1}<e_{2}$. Then we claim that $e_{1}$ is internally active with respect to $f$. Otherwise, there is a hypertree $g$ with $g(e_{1})=f(e_{1})-1$ and $g(e)=f(e)+1$ for some hyperedge $e<e_{1}$ and $g(e')=f(e')$ for all $e'\neq e,e_{1}$. We have that $g(e_{1})+g(e_{2})= t-2$, a contradiction. Hence there are still at most $|V|-t+1$ internally inactive hyperedges with respect to $f$.
\end{proof}

A bipartite graph $G=(V\cup E, \varepsilon)$ is said to be \emph{balanced} if $|V|=|E|$.

\begin{corollary}
Let $G=(V\cup E, \varepsilon)$ be a connected balanced bipartite graph with $|V|=|E|=n$. If $G$ has a 2-vertex cut $\{u_1,u_2\}$ in $E$ or $V$ so that $G-\{u_1,u_2\}$ has at least three connected components, then the coefficient of the term $x^{n-1}$ in $I_{G}(x)$ is 0.
\end{corollary}
\begin{proof}
It follows from Theorem \ref{thm 6} and the fact that taking abstract dual does not affect the interior polynomial.
\end{proof}

Similar to Theorem \ref{thm 6}, we can prove the following more general result, and the proof is left for the readers.

\begin{theorem}
Let $G=(V\cup E, \varepsilon)$ be a connected bipartite graph.  If it has $m$ pairs of vertex-disjoint 2-vertex cuts $\{\{v_{2i-1}, v_{2i}\}|i=1,2,\cdots,m\}$ in $V$ and $n$ pairs of vertex-disjoint 2-vertex cuts $\{\{e_{2i-1}, e_{2i}\}|i=1,2,\cdots,n\}$ in $E$, then the degree of the interior polynomial of $G$ is at most min\{$|E|-t+2m-1$, $|V|-k+2n-1$\}, where $t=\sum\limits_{i=1}^{m}t_{i}$, $k=\sum\limits_{i=1}^{n}k_{i}$, $t_{i}$ is the number of  connected components of $G-\{v_{2i-1}, v_{2i}\}$ for $i=1,2,\cdots,m$, and $k_{i}$ is the number of connected components of $G-\{e_{2i-1}, e_{2i}\}$ for $i=1,2,\cdots,n$.
\end{theorem}

\subsection{Monic property}
\noindent

In this subsection we consider the monic property of interior polynomials for balanced bipartite graphs.

\begin{lemma}\label{thm 7}
Let $G=(V\cup E, \varepsilon)$ be a connected balanced bipartite graph with $|V|=|E|=n$, and $E=\{e_1,e_2,\cdots,e_n\}$. Then the coefficient of the term $x^{n-1}$ of the interior polynomial $I_{G}(x)$ of $G$ is non-zero, if and only if each $f_i$ ($i=1,2,\cdots,n$) is a hypertree of $G$, where $f_i(e_i)=0$ and $f_i(e_j)=1$ for all $j\neq i$. Moreover, the coefficient of the term $x^{n-1}$ of the interior polynomial $I_{G}(x)$ of $G$ is in fact at most 1.
\end{lemma}
\begin{proof}
Without loss of generality, we assume that $e_{1}<e_{2}<\cdots<e_{n}$. The sufficiency follows from the fact that $f_{1}$ is a hypertree with internal inactivity $n-1$.
For necessity, if $f$ is a hypertree with internal inactivity  $n-1$, then $f$ has exactly one zero entry and $n-1$ one entries since $\sum\limits_{i=1}\limits^n f(e_i)=n-1$ and to ensue that internal inactivity of $f$ is $n-1$, each $f_i$ must be a hypertree of $G$.

Note that $f_2,\cdots, f_n$ and other hypertrees of $G$ with at least two zero entries have internal inactivity at most $n-2$, hence the coefficient of the term $x^{n-1}$ of the interior polynomial $I_{G}(x)$ for $G$ is at most  1.
\end{proof}

It is well-known that a graph is 2-connected if and only if it has an ear decomposition.

\begin{theorem}\label{dds}
Let $G=(V\cup E, \varepsilon)$ be a 2-connected balanced bipartite graph with $|V|=|E|=n$. If it has an ear decomposition such that each ear is from some element of $V$ to some element of $E$, then the coefficient of the term $x^{n-1}$ of the interior polynomial $I_{G}(x)$ for $G$ is 1.
\end{theorem}

\begin{proof}
Suppose that  $E=\{e_1,e_2,\cdots,e_n\}$. By Lemma \ref{thm 7}, we only need to show that each $f_{i}$ ($i=1,2,\cdots,n$) with $f_{i}(e_{i})=0$ and $f_{i}(e_j)=1$ for all $j\neq i$ is a hypertree of $G$.
We shall prove the theorem by induction on the number of ears in the ear decomposition of $G$.

If the number of ears is zero, then $G=C_{2n}$. Suppose $C_{2n}=v_1e_1v_2e_2\cdots v_ne_n$. Let $T_{i}$ be the spanning tree of $C_{2n}$ obtained from $C_{2n}$ by removing an edge incident with $e_{i}$ for every $i=1,2,\cdots,n$. Then the hypertree $f_{i}$ induced by $T_i$ satisfies $f_{i}(e_{i})=0$ and $f_{i}(e)=1$ for all $e\neq e_{i}$. Assume that the theorem holds for any 2-connected bipartite graph $G'=(V'\cup E', \varepsilon')$ whose ear decomposition has $k$ ears. Suppose that $G$ is obtained from $G'$ by adding the $k+1$-th ear $P$ with length $2l+1$ from some element of $V'$ to some element $e_{j}$ of $E'$. We order hyperedges of $E'$ firstly then $l$ hyperedges on $P$. By induction hypothesis, we suppose that $\tau'_{i}$ $(i=1,2,\cdots,|E|')$ is the spanning tree of $G'$ inducing $f'_{i}$, where $f'_{i}$ is the hypertree with $f'_{i}(e_{i})=0$ and $f'_{i}(e)=1$ for $e\in E'\setminus e_{i}$.

Now we show that $f_{i}$ $(i=1,2,\cdots,|E|'+l)$ with $f_{i}(e_{i})=0$ and $f_{i}(e)=1$ for $e\neq e_{i}$ is a hypertree of $G$. Let $\epsilon$ be the edge adjacent to $e_{j}$ in $P$. For $i=1,2,\cdots,|E|'$,  $\tau_i=\tau'_{i}+P-\epsilon$ is a spanning tree of $G$ and it exactly induces the hypertree $f_{i}$.  For $i=|E|'+1,|E|'+2,\cdots,|E|'+l$, $\tau_i=\tau'_{j}+P-\epsilon'_i$ is a spanning tree of $G$ and it exactly induces the hypertree $f_{i}$, where $\epsilon'_i$ is an edge adjacent to $e_{i}$ on $P$.

An example illustrating the proof is given in Figure 1.
\end{proof}
\begin{center}
\includegraphics[width=5.0in]{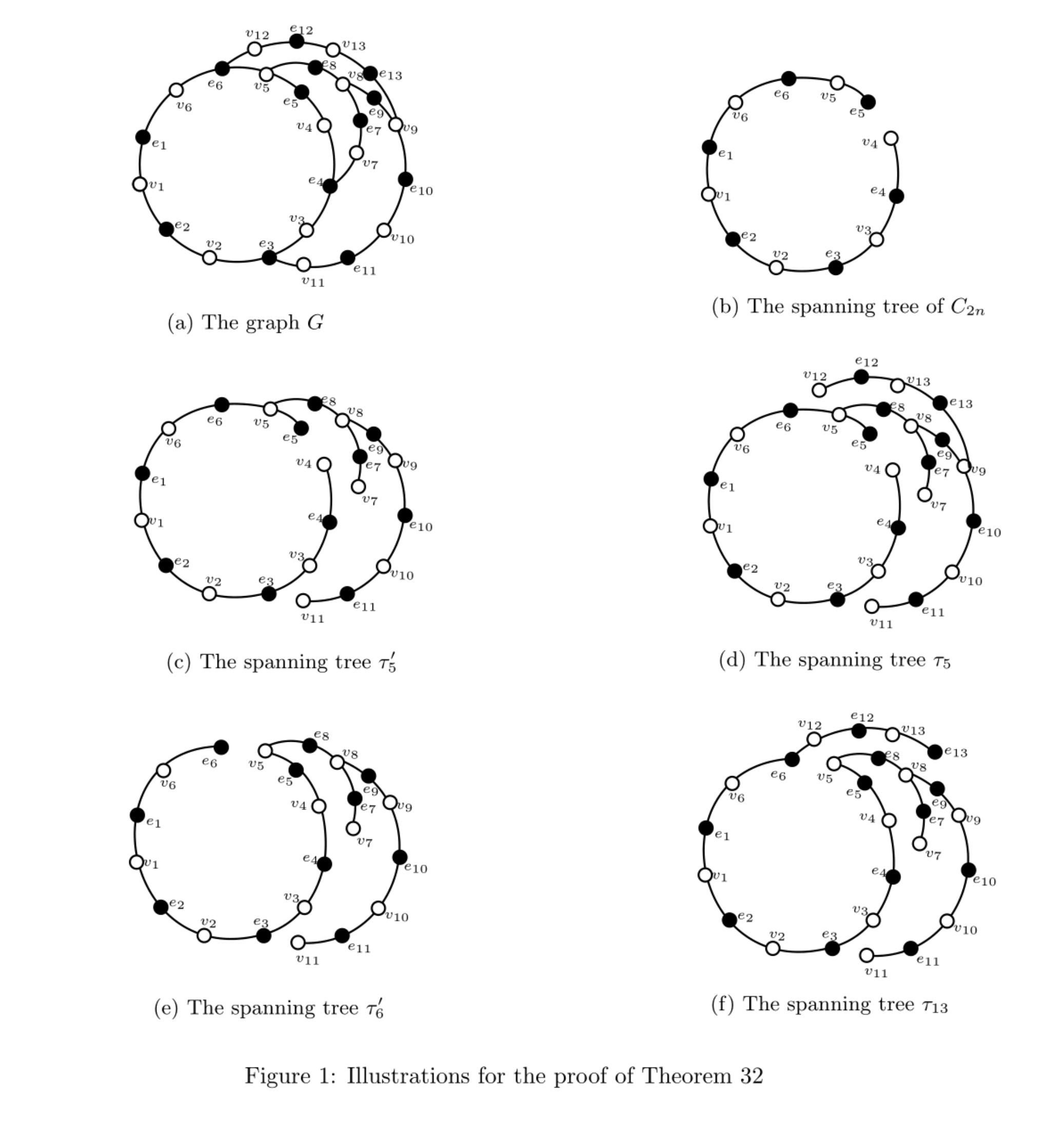}
\end{center}
\subsection{Linear representation}
\noindent

In this subsection, using deletion-contraction relations repeatedly, we show that the interior polynomial of any connected bipartite graph can be represented as a linear combination of the interior polynomials of connected balanced bipartite graphs.

\begin{theorem}[\cite{Kalman1}]\label{thm 8}
If $u\in V\cup E$ is a vertex of valence 1 in the bipartite graph $G$, then $I_{G}(x)=I_{G\setminus u}(x)$ and $X_{G}(y)=X_{G\setminus u}(y)$.
\end{theorem}

Let $G=(V\cup E, \varepsilon)$ be a connected bipartite graph and $e\in E$. We denote by $G\setminus e$ the bipartite graph obtained from $G$ by removing $e$ and all edges incident with $e$. Moveover, we denote by $G/e$ the bipartite graph obtained from $G\setminus e$ by identifying all vertices adjacent to $e$ and replacing all multi-edges by single ones.

\begin{theorem}[\cite{Kalman1}]\label{thm 9}
If $u\in V\cup E$ is a vertex of valence 2 in the bipartite graph $G$ and $G\setminus u$ is connected, then $I_{G}(x)=I_{G\setminus u}(x)+xI_{G/u}(x)$. If $e\in E$ and $d_{G}(e)=2$, then $X_{G}(y)=yX_{G\setminus e}(y)+X_{G/e}(y)$.
\end{theorem}

For a connected bipartite graph $G=(V\cup E, \varepsilon)$,
we use $G_{t}$ to denote the bipartite graph obtained from $G$ by adding $t$ vertices $V'=\{v_{1},v_{2},\cdots,v_{t}\}$ to $V$ and joining them to the same pair of vertices $e_{1},e_{2}\in E$ for some positive integer $t$ and we use $G'$ to denote the bipartite graph obtained from $G$ by identifying $e_{1}$ and $e_{2}$, and replacing all multi-edges by single ones. See Figure 2. We have the following Lemma \ref{thm 10}.

\begin{center}
\includegraphics[width=4.8in]{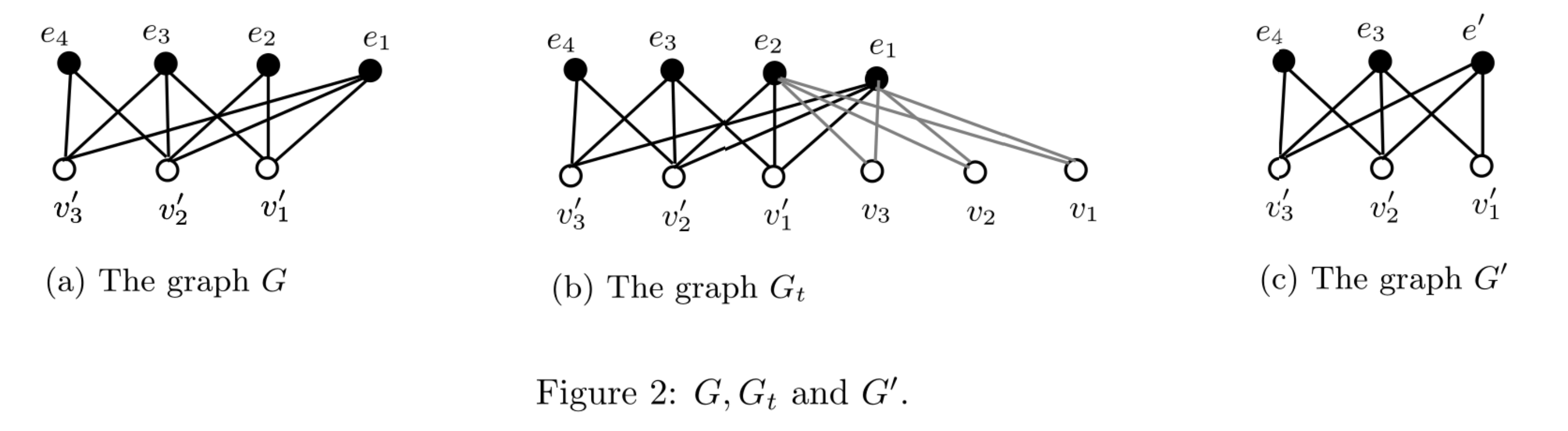}
\end{center}
\begin{lemma}\label{thm 10}
$I_{G}(x)=I_{G_{t}}(x)-txI_{G'}(x)$.
\end{lemma}
\begin{proof}
 Note that $G=G_{t}\setminus V'$  and  $I_{G_{t}/v_{1}}(x)=I_{G_{t}\setminus v_{1}/v_{2}}(x)=I_{G_{t}\setminus v_{1}\setminus v_{2}/v_{3}}(x)=\cdots=I_{G'}(x)$ by Theorem \ref{thm 8}. By Theorem \ref{thm 9}, we have that
$I_{G_{t}}(x)=I_{G_{t}\setminus v_{1}}(x)+xI_{G_{t}/v_{1}}(x)=I_{G_{t}\setminus v_{1}\setminus v_{2}}(x)+xI_{G_{t}\setminus v_{1}/v_{2}}(x)+xI_{G'}(x)=I_{G_{t}\setminus v_{1}\setminus v_{2}}(x)+2xI_{G'}(x)=\cdots=I_{G}(x)+txI_{G'}(x)$.  Then $I_{G}(x)=I_{G_{t}}(x)-txI_{G'}(x)$.
\end{proof}

\begin{theorem}[\cite{Kalman1}]\label{lem 10}
For any connected bipartite graph $G=(V\cup E, \varepsilon)$, the coefficient of the linear term in  the interior polynomial of $G$ is $n(G)=|\varepsilon|-(|V|+|E|)+1$.  
\end{theorem}

For a connected bipartite graph $G=(V\cup E, \varepsilon)$, if $G_{i}$ ($i=1,2$) is the bipartite graph obtained from $G$ by adding $m_{i}$ new vertices to $V$ and joining them to the same pair of vertices $e_{i},e'_{i}\in E$ for some positive integer $m_{i}$. Let $G^{i}$ be the bipartite graph obtained from $G$ by identifying $e_{i}$ and $e'_{i}$ and replacing all multi-edges by single ones for $i=1,2$. We know that $I_{G}(x)=I_{G_{1}}(x)-m_{1}xI_{G^{1}}(x)=I_{G_{2}}(x)-m_{2}xI_{G^{2}}(x)$ by Lemma \ref{thm 10}. Since the coefficient of the linear term in $I_{G^{i}}(x)$  is $n(G^{i})$ for $i=1,2$ by Theorem \ref{lem 10}, we know that if $m_{1}n(G^{1})=m_{2}n(G^{2})$, then the coefficients of the quadratic terms in the interior polynomials $I_{G_{1}}(x)$ and $I_{G_{2}}(x)$ are equal. In particular, we have:

\begin{corollary}
If $m_{1}=m_{2}$ and $|N_{G}(e_{1})\cap N_{G}(e'_{1})|=|N_{G}(e_{2})\cap N_{G}(e'_{2})|$, then the coefficients of the quadratic terms in the interior polynomials $I_{G_{1}}(x)$ and $I_{G_{2}}(x)$ are equal.
\end{corollary}

\begin{proof}
This is because $|N_{G}(e_{1})\cap N_{G}(e'_{1})|=|N_{G}(e_{2})\cap N_{G}(e'_{2})|$ will imply $n(G^{1})=n(G^{2})$.
\end{proof}

For a connected bipartite graph $G=(V\cup E, \varepsilon)$, without loss of generality, we assume that $|V|\leq|E|=n$. Let $t=|E|-|V|$.  Denote by  $G_{n}$ the bipartite graph obtained from $G$ by adding $t$ vertices to $V$ and joining them to the same pair of vertices $e_{1},e_{2}\in E$.  Denote by $G^{1}$ the bipartite graph obtained from $G$ by identifying $e_{1}$ and $e_{2}$ and replacing all multi-edges by single ones. Next, we recursively define two families of bipartite graphs as follows. Let $G_{n-i}=(E_{n-i}\cup V_{n-i}, \varepsilon_{n-i})$ be the bipartite graph obtained from $G^{i}=(E^{i}\cup V^{i}, \varepsilon^{i})$ by adding $t-i$ vertices to $V^{i}$ and joining them to the same pair of vertices $e^{i}_{1},e^{i}_{2}\in E^{i}$ for any $i=1,2,\cdots,t$. Let $G^{i+1}$ be the bipartite graph obtained from $G^{i}$ by identifying $e^{i}_{1}$ and $e^{i}_{2}$ and replacing all multi-edges by single ones for any $i=1,2,\cdots,t-1$.

\begin{corollary}
With notations above, we have $I_{G}(x)=\sum\limits_{i=0}^{t}\frac{t!}{(t-i)!}(-1)^{i}x^{i}I_{G_{n-i}}(x)$.
\end{corollary}
\begin{proof}
If $t=0$, then the corollary is obvious. If $t\neq0$, then by Lemma \ref{thm 10}, we have that  $I_{G^{i}}(x)=I_{G_{n-i}}(x)-(t-i)xI_{G^{i+1}}(x)$ for any $i=1,2,\cdots,t$. Combining with $I_{G}(x)=I_{G_{n}}(x)-txI_{G^{1}}(x)$, we have $I_{G}(x)=\sum\limits_{i=0}^{t}\frac{t!}{(t-i)!}(-1)^{i}x^{i}I_{G_{n-i}}(x)$.
\end{proof}

Note that $|V_{n-i}|=|E_{n-i}|=n-i$ for all $i=1,2,\cdots,t$ by the construction of $G_{n-i}$, we have:
\begin{theorem}
The interior polynomial of any connected bipartite graph can be written as a linear combination of the interior polynomials of connected balanced bipartite graphs.
\end{theorem}

\section{The exterior polynomial}
\noindent

In this section, we study interpolatory property  and the coefficient of the linear term of the exterior polynomial.

Without loss of generality, we assume that vertices of $E$ will be regarded as hyperedges for a connected bipartite graph $G=(V\cup E, \varepsilon)$ in this section.

\subsection{Interpolatory property}
\noindent

In this subsection, we will show that the exterior polynomial of any connected bipartite graph is interpolating. We also need some lemmas similar to the interior polynomial.

\begin{lemma}\label{lem 7}
Let $G=(V\cup E, \varepsilon)$ be a connected bipartite graph and $f$ be a hypertree of $G$. Given an order on $E$, then the hyperedge $e$ is externally inactive with respect to $f$ if and only if there exists a hyperedge $e'<e$ so that  $f(e')\neq 0$ and every subset of $E$, which contains $e$ and does not contain $e'$, is not tight at $f$.
\end{lemma}
\begin{proof}
The hyperedge $e$ is externally inactive with respect to  $f$, if and only if for  some hyperedge $e'<e$, $e'$ can  transfer  valence to $e$ with respect to  $f$. We know that $e'$ can transfer  valence to $e$ with respect to  $f$ if and only if $f(e')\neq 0$, and every subset $E'\subset  E$, which contains $e$ and does not contain $e'$, is not tight at $f$ by Lemma \ref{lem 20}, which completes the proof.
\end{proof}

Let $G=(V\cup E, \varepsilon)$ be a connected bipartite graph, $e_{1}$ and $e_{2}$ be hyperedges of $G$ and $e_{1}\neq e_{2}$. Let $f_{1}$ and $f_{2}$ be hypertrees of $G$ with $f_{1}(e_{1})>f_{2}(e_{1})$ and $f_{1}(e)=f_{2}(e)$ for all $e\in E$ and $e\neq e_{1},e_{2}$. The following property in Lemma \ref{lem 8} is true for two hypertrees $f_{1}$ and $f_{2}$.

\begin{lemma}\label{lem 8}
Given an order on $E$, if a hyperedge $e>e_{1}$  is externally active with respect to $f_{1}$, then it is also externally active with respect to $f_{2}$.
\end{lemma}

\begin{proof}
If the hyperedge $e$  is externally active with respect to  $f_{1}$, then for any  hyperedge $e'<e$, (i) $f_1(e')=0$ or (ii)  there exists a subset $U'\subset  E$, which contains $e$ and does not contain $e'$,  is  tight at $f_{1}$ by Lemma \ref{lem 7}. We first claim that $e_{2}$ cannot transfer valence to $e$ for $f_{2}$. Otherwise, by Lemma \ref{lem 4}, $e_{1}$ can transfer valence to $e$ for $f_{1}$, which implies that the hyperedge $e$  is externally inactive with respect to  $f_{1}$ as $e>e_{1}$, a contradiction. If $f_{1}(e')=0$ ($e'$ can not be $e_1$) and $e'\neq e_{2}$, then $e'$ cannot transfer valence to $e$ for $f_{2}$ since $f_{2}(e')=f_{1}(e')=0$. If (ii) is true, then in particular, there exists a subset $U_{1}\subset  E$, which contains $e$ and does not contain $e_{1}$,  is  tight at $f_{1}$, since $e_{1}<e$. Take $U=U'\cap U_{1}$ (it is possible that $U'=U_{1}$). We have that  the subset $U\subset  E$, which contains $e$, and does not contain $e'$ and $e_{1}$,  is  tight at $f_{1}$ by Theorem \ref{thm 2}. Note that
$\sum\limits_{e\in U}f_{1}(e)\leq \sum\limits_{e\in U}f_{2}(e)$. We have that $U$ is  tight at $f_{2}$. It implies that $e'$ cannot transfer valence to $e$ for $f_{2}$. Thus, $e$  is externally active with respect to  $f_{2}$.
\end{proof}

\begin{lemma}\label{lem 9}
Let $G=(V\cup E, \varepsilon)$ be a connected bipartite graph. If there exists  a hypertree with external inactivity $k$ for $G$,  then there exists  a hypertree with external inactivity $k-1$ in $G$ for any integer $k\geq1$.
\end{lemma}

\begin{proof}
Given an order on $E$, assume $f_{1}$ is a hypertree with external inactivity $k$ in $G$ and $e_{m}$ is the smallest hyperedge that is externally inactive with respect to $f_{1}$. Let $f$ be a hypertree with external inactivity $k$  and the hyperedge $e_{m}$ is the smallest externally inactive hyperedge  with respect to $f$ so that the entry $f(e_{m})$ of $e_{m}$ is biggest. Let $e_{n}$ be the smallest hyperedge that can transfer valence to $e_{m}$. Then there is the hypertree $g$ with $g(e_{m})=f(e_{m})+1$, $g(e_{n})=f(e_{n})-1$ and $g(e)=f(e)$ for all $e\neq e_{m},e_{n}$.  Next we prove that the hyperedge $e$ is externally active with respect to $g$ if and only if $e$ is externally active with respect to $f$ for every hyperedge $e\neq e_{m}$.

For any hyperedge $e$ with $e<e_{n}$, since $e$ is externally active with respect to $f$, for every  hyperedge $e'<e$, (i)  $f(e')=0$ or (ii)  there exists a subset $U'\subset E$, which contains $e$ and does not contain $e'$,  is  tight at $f$  by Lemma \ref{lem 7}. If (i) holds, then $g(e')=f(e')=0$, clearly, $e'$  cannot transfer valence to $e$ with respect to  $g$. Assume that (ii) is true. Since $e'$  can not transfer valence to $e_{m}$ with respect to $f$ there exists $E_{2}\subset E$, which contains $e_{m}$ and does not contain $e'$, is tight at $f$ by Lemma \ref{lem 20}. Take $E'=U'\cup E_{2}$.  Then $E'$, which contains $e_{m}$ and $e$ and does not contain $e'$,  is  tight at $f$ by Theorem \ref{thm 2}. Note that $\sum\limits_{e\in E'}f(e)\leq \sum\limits_{e\in E'}g(e)$. We have that  $E'$ is  tight at $g$. Then for every hyperedge $e'<e$, $e'$  cannot transfer valence to $e$ with respect to  $g$. Thus,  $e$ is externally active with respect to $g$.

For the hyperedge $e_{n}$, we claim $e_{n}$ is externally active with respect to $g$. Assume that the opposite is true. Then for some hyperedge $e'<e_{n}$, $e'$ can  transfer valence to $e_{n}$ with respect to $g$. Then there exists the hypertree $f'$ satisfying $f'(e_{m})=f(e_{m})+1$ and $f'(e')=f(e')-1$ and $f'(e)=f(e)$ for all $e\in E$ and $e\neq e_{m},e_{n}$, that is, $e'$ can  transfer valence to $e_{m}$ with respect to $f$. This  contradicts  the choice of the hyperedge $e_{n}$.

For any externally active hyperedge  $e>e_{n}$ with respect to $f$ (in fact, $e$ is externally active if $e_{n}<e<e_{m}$ with respect to $f$), $e$ is also externally active with respect to $g$ by Lemma \ref{lem 8}.

For any externally inactive hyperedge $e>e_{m}$  with respect to $f$, we claim that $e$ is externally inactive with respect to $g$. Otherwise, $e$ is externally active with respect to $f$ by Lemma \ref{lem 8}, a contradiction.

It is obvious that $e_{m}$ is externally active with respect to $g$ by the choice of the hypertree $f$. In fact, if $e_{m}$ is externally inactive with respect to $g$, then  $g$ is a hypertree with externally inactive $k$ and the hyperedge $e_{m}$ is the smallest externally inactive hyperedge  with respect to $f$.  Since $f(e_{m})<g(e_{m})$, it contradicts the choice of the hypertree $f$. Thus, $g$ is a hypertree with externally inactive $k-1$.
\end{proof}
\begin{theorem}
The exterior polynomial of any connected bipartite graph is interpolating.
\end{theorem}
\begin{proof}
It follows directly from Lemma \ref{lem 9}.
\end{proof}

\subsection{Coefficient of linear term}
\noindent

In this subsection, we will study the coefficient of the linear term of the exterior polynomial by a similar argument in the interior polynomial.

Before studying the coefficient of the linear term of the exterior polynomial, we firstly recall the coefficient of constant term of the interior polynomial and the exterior polynomial.

\begin{theorem}[\cite{Kalman1}]\label{thm 16}
For any connected bipartite graph, the coefficient of the constant term of the interior polynomial and the exterior polynomial are $1$.
\end{theorem}

Let $G=(V\cup E, \varepsilon)$ be a connected bipartite graph. Given an arbitrary order on $E$ ($e_{1}<e_{2}<\cdots<e_{|E|}$), let $G_{k}=(E_{k}\cup V_{k}, \varepsilon_{k})$ be the bipartite graph formed by $E_{k}=\{e_{k},e_{k+1},\cdots,e_{|E|}\}$, all edges of $G$ incident with some element of $E_{k}$, and their endpoints in $V$ for each $k=1,2,\cdots,|E|$. Clearly, $G_{1}=G$ and $G_{|E|}$ is a star. Let $n(G_{k})$ be the nullity of the graph $G_{k}$, that is, $n(G_{k})=|\varepsilon_{k}|-(|V_{k}|+|E_{k}|)+c_{k}$, where $c_{k}$ is the number of connected components of $G_{k}$. Then $n(G_{1})=n(G)$ and $n(G_{|E|})=0$. Let $n(e_{k})=n(G_{k})-n(G_{k+1})$ for any $k=1,2,\cdots,|E|-1$ and $n(e_{|E|})=0$. Moreover, a hypertree $g'$ of $G$ is defined as $g'(e)=d_{G}(e)-1-n(e)$ for all $e\in E$. By the proof of Theorem \ref {thm 16} in \cite{Kalman1}, we know that $E_{k}$ is tight at $g'$ for any $k=1,2,\cdots,|E|$. Moreover, $g'$ is the unique hypertree with external inactivity 0, called \emph{the exterior  greedy hypertree} of $G$.

\begin{theorem}\label{thm 11}
For a connected bipartite graph $G=(V\cup E, \varepsilon)$, if  $G-e$ is connected for each $e\in E$, then the coefficient of the linear term of  the exterior polynomial is $|V|-1$.
\end{theorem}
\begin{proof}
After arbitrarily fixing an order on $E$ ($e_{1}<e_{2}<\ldots<e_{m}<\ldots<e_{|E|}$), let $g'$ be the exterior greedy hypertree above under this order.

Firstly, we show that for any hyperedge $e_{m}$ with $g'(e_{m})\neq 0$ and for any $0\leq i \leq g'(e_{m})$, there is an $i$-element multiset consisting of  hyperedges smaller than $e_{m}$ so that the result of reducing  $g'(e_{m})$ to $g'(e_{m})-i$, and increasing the $g'$-value to each of its element of the $i$-multiset by the multiplicity of the element, is also a hypertree of $G$.

Clearly, for any hyperedge $e_{m}$, the conclusion is true for $i=0$. Assume that the opposite is true, that is, there is no such a hypertree for some $e_{m}$ and some $1\leq i \leq g'(e_{m})$. Without loss of generality, we assume that  there is a hypertree $g'_{j}$ satisfying conditions  for some $0\leq j\leq g'(e_{m})-1$ and there is no  hypertree satisfying conditions  for $j+1$. Then the hyperedge $e_{m}$ is internally active with respect to  $g'_{j}$. Since $g'_{j}(e_{m})\neq 0$, for any $e_{x}<e_{m}$, there is a subset $S_{x}\subset E$ , which contains $e_{x}$ and does not contain $e_{m}$, is tight at $g'_{j}$ by Lemma \ref{lem 3}. Note that $g'_{j}(e_{i})=g'(e_{i})$ for all $e_{i}>e_{m}$. We have that $E_{m+1}=\{e_{m+1},\cdots,e_{|E|}\}$ is tight at $g'_{j}$. Put $S=\bigcup\limits_{e_{x}<e_{m}}S_{x} \cup E_{m+1}=E\setminus \{e_{m}\}$.  Then $S$ is tight at $g'_{j}$ by Theorem \ref{thm 2}. Let $\tau$ be a spanning tree of $G$ inducing $g'_{j}$. Then $\tau|_{S}$ should be a spanning forest of $G|_{S}$ by Theorem \ref{thm}. Because $G-e$ is connected for any $e\in E$, we have that $G|_{S}=G-e_{m}$ is connected. Then $\tau|_{S}$ is a spanning tree of $G-e_{m}$, which implies that it is $\tau-e_{m}$. Since $g'(e_{m})>0$, $d_{\tau}(e_{m})\geq 2$. This implies that $\tau-e_{m}$ is disconnected, a contradiction. In addition, $g'(e_1)$ must be zero, otherwise, $G-e_1$ should not be connected.

Next we are going to construct $|V|-1$ such hypertrees as follows. Given the hyperedge $e_{m}$ with $g'(e_{m})\neq 0$ and $i$ ($1\leq i \leq g'(e_{m})$), we define $f_{e_{m},i}$: $E \rightarrow N$ be the hypertree so that its associated multiset $M_{e_{m},i}$ is the largest in reverse order, that is, if $f\neq f_{e_{m},i}$ is another hypertree (recall that $f$ is obtained by setting $f(e_m)=g'(e_{m})-i$ and increasing the $g'$-value to each of its element of one $i$-multiset consisting of  hyperedges smaller than $e_{m}$ by the multiplicity of the element) and $e'=\max\{e_{j}|f_{e_{m},i}(e_{j})\neq f(e_{j})\}$, then $f_{e_{m},i}(e')> f(e')$.
Since $\sum\limits_{e\in E}g'(e)=|V|-1$, there are $|V|-1$ those hypertrees.

In the following we show that $e_{m}$ is a unique externally inactive hyperedge with respect to $f_{e_m,i}$.

\begin{enumerate}
\item[(1)] If $e_{j}>e_{m}$, then $e_{j}$ is  externally active with respect to $f_{e_m,i}$. This is because $f_{e_{m},i}(e)=g'(e)$ for any $e>e_{m}$. $E_{j}$ is tight at $g'$ and also at $f_{e_m,i}$, one can not transfer valence from
$E\backslash E_{j}$ to $e_{j}$ with respect to $f_{e_m,i}$.

\item[(2)] If $e_{j}<e_{m}$, then  $e_{j}$ is  externally active   with respect to $f_{e_{m},i}$ due to choice of $f_{e_{m},i}$. Otherwise, there is the hypertree $g_{1}$  with $g_{1}(e_{j})=f_{e_{m},i}(e_{j})+1$, $g_{1}(e_{x})=f_{e_{m},i}(e_{x})-1$ for some $e_{x}<e_{j}$ and $g_{1}(e')=f_{e_{m},i}(e')$ for $e'\neq e_{j},e_{x}$. It contradicts the choice of $f_{e_{m},i}$.

\item[(3)] Since $E_{m}$ is tight at $g'$, it is not tight at $f_{e_{m},i}$. $E_{m+1}$ is tight both at $g'$ and $f_{e_{m},i}$.
So it is possible to transfer valence from $E\backslash E_{m}$ to $e_{m}$ for $f_{e_{m},i}$ and hence $e_{m}$ is externally inactive with respect to $f_{e_{m},i}$.
\end{enumerate}

Finally, we prove that if $f$ is a hypertree with a unique externally inactive hyperedge $e_{m}$, then $f$ is one of the $f_{e_{m},i}$'s.
\vskip0.2cm

\noindent\textbf{Claim 1.} For all $e_{x}>e_{m}$, $E_{x}=\{e_{y}\in E|e_{y}\geq e_{x}\}$ is tight at the hypertree $f$.
Assume that \textbf{Claim 1} is not true, that is,  $E_{x}=\{e_{y}\in E|e_{y}\geq e_{x}\}$ is not tight at  $f$ for some $e_{x}>e_{m}$. Then $f$ is a hypertree such that it is possible to transfer valent from element of $E\setminus E_{x}$ to some element of $E_{x}$ by Lemma \ref{lem 1}, that is, some element of $E_x$ is externally inactive  with respect to $f$, a contradiction. Thus, \textbf{Claim 1} is true.

\textbf{Claim 1} implies that $f(e_{y})=g'(e_{y})$ for all $e_{y}>e_{m}$. Since $e_{m}$ is unique externally inactive hyperedge with respect to $f$ we have $f(e_{m})<g'(e_{m})$. Let $i=g'(e_{m})-f(e_{m})$. Then the following claim holds.

\vskip0.2cm
\noindent\textbf{Claim 2.}  $f=f_{e_{m},i}$.

Assume that $f\neq f_{e_{m},i}$.  Note that  $f_{e_{m},i}(e_{x})=g'(e_{x})$ for all $e_{x}>e_{m}$ by the construction of $f_{e_{m},i}$. It implies that $f_{e_{m},i}(e_{x})=f(e_{x})$ for all $e_{x}\geq e_{m}$.  Assume that $e_{y}=\max\{e_{x}|f_{e_{m},i}(e_{x})\neq f(e_{x})\}$. There are two cases.

\noindent{\bf Case 1.} $f_{e_{m},i}(e_{y})> f(e_{y})$.

Let $S=\{e_{x}|f(e_{x})\neq 0$ and  $e_{x}<e_{y}\}$. Since $e_{y}$ is externally active  with respect to $f$, there is a subset $E_{x}$$\subset E$, which contains $e_{y}$ and does not contain $e_{x}$,  is tight at  $f$ for any $e_{x}\in S$ by Lemma \ref{lem 7}. Let $E'=\bigcap\limits_{e_{x}\in S}E_{x}$. Then $E'$ is tight at  $f$ by Theorem \ref {thm 2}, that is, $\sum\limits_{e_{x}\in E'}f(e_{x})=\mu(E')$. Note $e_{y}\in E'$, $E'\cap S=\emptyset$ and $f_{e_{m},i}(e_{y})>f(e_{y})$. Moreover, $f_{e_{m},i}(e_{x})=f(e_{x})$ or $f(e_{x})=0$ for any $e_{x}\in E'$. We have that
$\sum\limits_{e_{x}\in E'}f_{e_{m},i}(e_{x})>\sum\limits_{e_{x}\in E'}f(e_{x})=\mu(E')$, a contradiction.

\noindent{\bf Case 2.} $f_{e_{m},i}(e_{y})< f(e_{y})$.
Let $S'=\{e_{x}|f_{e_{m},i}(e_{x})\neq 0$ and  $e_{x}<e_{y}\}$. Since $e_{y}$ is externally active with respect to $f_{e_{m},i}$, there is a subset $E'_{x}$$\subset E$, which contains $e_{y}$ and does not contain $e_{x}$,  is tight at  $f_{e_{m},i}$ for any $e_{x}\in S'$ by Lemma \ref{lem 7}. Let $E''=\bigcap\limits_{e_{x}\in S'}E'_{x}$. Then $E''$ is tight at  $f_{e_{m},i}$ by Theorem \ref {thm 2}, that is, $\sum\limits_{e_{x}\in E''}f_{e_{m},i}(e_{x})=\mu(E'')$. Note  $f(e_{y})>f_{e_{m},i}(e_{y})$, and $f(e_{x})=f_{e_{m},i}(e_{x})$ or $f_{e_{m},i}(e_{x})=0$ for any $e_{x}\in E''$. We have that
$\sum\limits_{e_{x}\in E''}f(e_{x})>\sum\limits_{e_{x}\in E''}f_{e_{m},i}(e_{x})=\mu(E'')$, a contradiction.
Hence, \textbf{Claim 2} is true.

Thus, the number of hypertrees with external inactivity 1 is exactly $|V|-1$.
\end{proof}

\section{Examples}
\noindent

In this section, as examples we shall compute the interior polynomial and the exterior polynomial for several families of bipartite graphs. Computational results in this section are consistent with results in Sections 3 and 4.

\begin{observation}\label{obs 2}
If $G=(V\cup E, \varepsilon)$ is a tree, then $I_{G}(x)=X_{G}(y)=1$.
\end{observation}

\begin{proof}
Since there is only one spanning tree for a tree there is only one hypertree $f$ for $G$. Furthermore, every hyperedge $e\in E$ will be internally active and externally active with respect to the hypertree $f$ since there are no other hypertrees. So $I_{G}(x)=x^0=1$ and $X_{G}(y)=y^0=1$.
\end{proof}

\begin{theorem}\label{obs 1}
If $G=(V\cup E, \varepsilon)$ is a cycle of length $2n$, then $I_{G}(x)=1+x+x^{2}+\cdots+x^{n-1}$ and $X_{G}(y)=1+(n-1)y$.
\end{theorem}

\begin{proof}
Since the two hypergraphs induced by $C_{2n}$ are the same, it is justified the exterior polynomial of $C_{2n}$. Recall that $\sum\limits_{e\in E}f(e)=n-1$ for any hypertree $f$ of $G$. Moreover, $0\leq f(e)\leq d_{G}(e)-1$  and $d_{G}(e)=2$ for all $e\in E$. Then given an order $e_{1}<e_{2}<\cdots<e_{n}$ on $E$, there are at most $n$ hypertrees $f_{1}=(0,1,\cdots,1)$, $f_{2}=(1,0,\cdots,1)$,$\cdots$,$f_{n}=(1,1,\cdots,0)$ for  $G$. Moreover, let the subgraph $\tau_{i}$ be the bipartite graph obtained from $G$ by removing an edge incident with $e_{i}$ for every $i=1,2,\cdots,n$. It is obvious that $\tau_{i}$ is a spanning tree inducing $f_{i}$ of $G$ for all $i=1,2,\cdots,n$.  We obtain that there are exactly $n$ hypertrees $f_{i}$ $(i=1,2,\cdots,n)$ of $G$. Clearly,  $f_{i}$ is a hypertree with internal inactivity $n-i$ in $G$ for any $i=1,2,\cdots,n$, $f_{1}$ is a hypertree with external inactivity $0$, and $f_{i}$ is a hypertree with external inactivity $1$ for any $i=2,3,\cdots,n$. Thus, $I_{G}(x)=1+x+x^{2}+\cdots+x^{n-1}$ and $X_{G}(y)=1+(n-1)y$.
\end{proof}
\begin{corollary}
 If $G=(V\cup E, \varepsilon)$ is unicyclic and the length of the unique cycle is $2n$, then $I_{G}(x)=1+x+x^{2}+\cdots+x^{n-1}$ and $X_{G}(y)=1+(n-1)y$.
\end{corollary}

\begin{proof}
It follows from Theorem \ref {thm 5} (1), Observation \ref {obs 2} and Theorem \ref {obs 1}.
\end{proof}

Let $P_{n+1}$ be the path of length $n$. Let $G\times H$ be the Cartesian product of $G$ and $H$.

\begin{corollary}
If $G=(V\cup E,\varepsilon)$ is  $P_{n+1}\times P_{2}$, then $I_{G}(x)=(1+x)^{n}$ and $X_{G}(y)=(1+y)^{n}$.
\end{corollary}
\begin{proof}
Note that $G$ can be obtained by identified edges of $C_4$.  See Figure 3. So it follows form Theorem \ref {thm 5} (2) and Theorem \ref {obs 1}.
\end{proof}
\begin{center}
\includegraphics[width=3.3in]{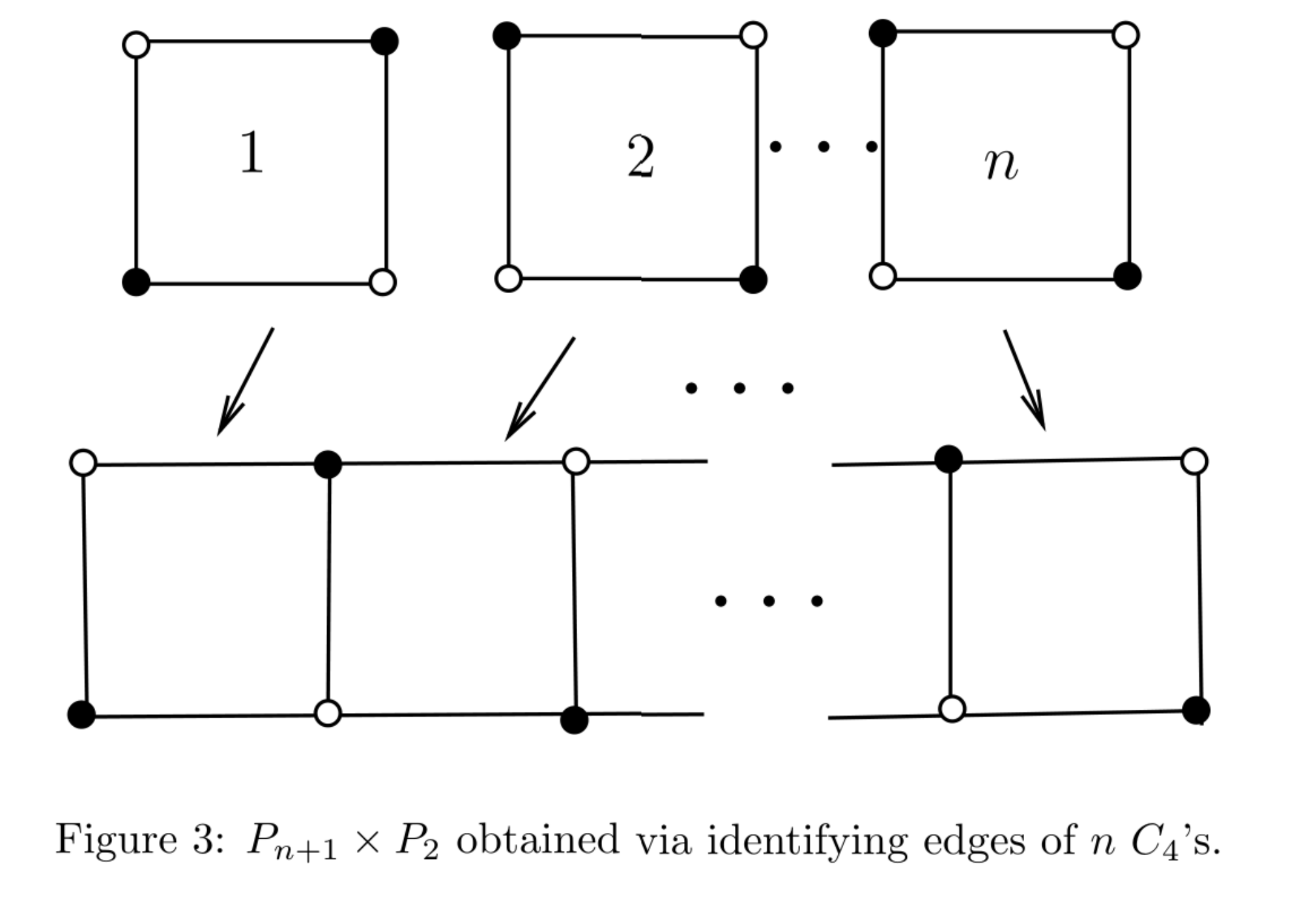}
\end{center}

In the following we consider interior and exterior polynomials of some dense graphs. Beforehand, we recall two combinatorial numbers as follows (see \cite{West}, p.19-20).

\begin{lemma}[\cite{West}]\label{lem 30}
The number of ordered partitions of $n$ into $k$ parts is $\binom{n-1}{k-1}$.
\end{lemma}

\begin{lemma}[\cite{West}]\label{lem 31}
The number of $m$-element multisets of an $n$-element set is $\binom{m+n-1}{n-1}$.
\end{lemma}

\begin{theorem}
Let $G=(V\cup E, \varepsilon)$ be the complete bipartite graph $K_{m,n}$ $(m\leq n)$. Then
\begin{enumerate}
\item[(1)] $I_{G}(x)=\sum\limits_{i=0}^{m-1}\binom{n-1}{i}\binom{m-1}{i}x^{i}$.
\item[(2)] $X_{G}(y)=\sum\limits_{i=0}^{n-1}\binom{m+i-2}{i}y^{i}$ if we regard $E$ as hyperedges.
\end{enumerate}
\end{theorem}

\begin{proof}
If $m=1$, then $G$ is a star. It is clear that the conclusion is true from Observation \ref {obs 2}.

Assume $m\geq 2$. By Theorem \ref {thm 3}, the interior polynomial and the exterior polynomial are independent of the chosen order of the hyperedges. Without loss of generality, we assume that $|V|=m$ and $|E|=n$. Moveover, let $v_{1}<v_{2}<\cdots<v_{m}$ and $e_{1}<e_{2}<\cdots<e_{n}$.

We have that the interior polynomials of abstract dual hypergraphs are equal by Theorem \ref {thm 4}. Without loss of generality we consider vertices of $E$ as hyperdeges.
Firstly, we prove that if a function $f$ satisfies $\sum\limits_{e\in E}f(e)=m-1$ and $0\leq f(e)\leq m-1$ for all $e\in E$, then $f$ is a hypertree of $K_{m,n}$.

We assume that the cardinality of the set $\{e|f(e)\neq 0\}$ is $i$ (note $0\leq f(e)\leq m-1$ and $\sum\limits_{e\in E}f(e)=m-1$, so $i>0$) and without loss of generality, suppose that $f(e_{1})\neq0$,$\cdots$,$f(e_{i})\neq0$. We construct the spanning tree $T$ of $G$ inducing $f$ as follows: $e_{1}$ is connected to $v_{1}$, $v_{2}$,$\cdots$, $v_{f(e_{1})+1}$,$\cdots$, and $e_{k}$ is connected to $v_{[\sum\limits_{j=1}^{k-1}f(e_{j})]+1}$, $v_{[\sum\limits_{j=1}^{k-1}f(e_{j})]+2}$,$\cdots$, $v_{[\sum\limits_{j=1}^{k}f(e_{j})]+1}$,\ldots, and $e_{i}$ is connected to $v_{[\sum\limits_{j=1}^{i-1}f(e_{j})]+1}$, $v_{[\sum\limits_{j=1}^{i-1}f(e_{j})]+2}$,$\cdots$, $v_{[\sum\limits_{j=1}^{i}f(e_{j})]+1}=v_{m}$, finally each $e\in E\backslash\{e_{1},e_{2},\cdots,e_{i}\}$ is connected to $v_{1}$.

Next, we consider hypertrees $f$ with $i$ positions that are not 0 for $i=1,2,\cdots,m-1$. If $f(e_{1})=0$, then the hypertree $f$ is a hypertree with internal  inactivity  $i$ and there are $\binom{n-1}{i}\binom{m-2}{i-1}$ such hypertrees by Lemma \ref {lem 30}.  If $f(e_{1})\neq0$, then   hypertree $f$ is a hypertree with internal  inactivity  $i-1$ and there are $\binom{n-1}{i-1}\binom{m-2}{i-1}$ such hypertrees by Lemma \ref {lem 30}, that is, there are $\binom{n-1}{i}\binom{m-2}{i-1}$  hypertrees with internal  inactivity $i$ and $\binom{n-1}{i-1}\binom{m-2}{i-1}$  hypertrees with internal  inactivity $i-1$ for  hypertrees with $i$ positions that are not 0. Note that $\binom{n-1}{0}\binom{m-2}{0-1}=0$ and $\binom{n-1}{m-1}\binom{m-2}{m-1}=0$. We have that
\begin{eqnarray*}
I_{G}(x)&=&\sum_{i=1}^{m-1}\binom{n-1}{i}\binom{m-2}{i-1}x^{i}+\sum_{i=1}^{m-1}\binom{n-1}{i-1}\binom{m-2}{i-1}x^{i-1}\\
&=&\sum_{i=1}^{m-1}\binom{n-1}{i}\binom{m-2}{i-1}x^{i}+\sum_{i=0}^{m-2}\binom{n-1}{i}\binom{m-2}{i}x^{i}\\
&=&\sum_{i=1}^{m-1}\binom{n-1}{i}\binom{m-2}{i-1}x^{i}+\binom{n-1}{0}\binom{m-2}{0-1}x^{0}\\
&&+\sum_{i=1}^{m-2}\binom{n-1}{i}\binom{m-2}{i}x^{i}+\binom{n-1}{m-1}\binom{m-2}{m-1}x^{m-1}\\
&=&\sum_{i=0}^{m-1}\binom{n-1}{i}\binom{m-2}{i-1}x^{i}+\sum_{i=0}^{m-1}\binom{n-1}{i}\binom{m-2}{i}x^{i}\\
&=&\sum_{i=0}^{m-1}\binom{n-1}{i}\binom{m-1}{i}x^{i}.
\end{eqnarray*}

For the exterior polynomial $X_{G}(y)$ of $G$, we assume that
vertices of $E$ are hyperdeges. For each $i=0,1,\cdots,n-1$, we consider a hypertree $f$ with $f(e_{i+1})\neq 0$ and $f(e_{j})=0$ for all $j=1,\cdots,i$. Then the hypertree $f$ is a hypertree with external  inactivity  $n-i-1$. If $f(e_{i+1})=j$ for $j=1,2,\cdots,m-1$, then there are $\binom{m-j+n-i-3}{n-i-2}$ such hypertrees by Lemma \ref {lem 31}.  Then  there are $\sum\limits_{j=1}^{m-1}\binom{m-j+n-i-3}{n-i-2}=\binom{m+n-i-3}{n-i-1}$  hypertrees with external  inactivity  $n-i-1$, that is, there exist $\binom{m+i-2}{i}$ hypertrees with external  inactivity $i$. Thus, we have that $X_{G}(y)=\sum\limits_{i=0}^{n-1}\binom{m+i-2}{i}y^{i}$.
\end{proof}

\begin{theorem}
Let $G=(V\cup E,\varepsilon)$ be a bipartite graph obtained from $K_{m,n}(m\leq n)$ by deleting a matching with $q$ ($q\leq m$) edges. Then
\begin{enumerate}
\item[(1)] $I_{G}(x)=1+[(n-1)(m-1)-q]x+\sum\limits_{i=2}^{m-1}\binom{n-1}{i}\binom{m-1}{i}x^{i}$.
\item[(2)] $X_{G}(y)=\sum\limits_{i=0}^{n-2}\binom{m+i-2}{i}y^{i}+[\binom{m+n-3}{n-1}-q]y^{n-1}$ if we regard $E$ as hyperedges.
\end{enumerate}
\end{theorem}

\begin{proof}
Without loss of generality, we assume that $v_{1}<v_{2}<\cdots<v_{m}$ and $e_{1}<e_{2}<\cdots<e_{n}$ in $V$ and $E$, respectively, and $M=\{v_{i}e_{i}|i=1,2,\cdots,q\}$ and $E'=\{e_{i}|i=1,2,\cdots,q\}$.

A similar argument of the proof of complete bipartite graphs shows that if a function $f$ satisfies $\sum\limits_{e\in E}f(e)=m-1$, $0\leq f(e)\leq m-2$ for all $e\in E'$ and $0\leq f(e)\leq m-1$ for all $e\in E \setminus E'$, then $f$ is a hypertree of $G$.

We assume that the cardinality of the set $\{e|f(e)\neq 0, e\in E'\}$ ($\{e|f(e)\neq 0, e\in E\setminus E'\}$) is  $i$ ($i'$, resp.) and $f(e_{1})\neq0$,$\cdots$,$f(e_{i})\neq0$,$f(e_{q+1})\neq0$,$\cdots$,$f(e_{q+i'})\neq0$. Then the spanning tree of $G$ inducing $f$ can be constructed as follows: $e_{1}$ is connected to $v_{2}$, $v_{3}$,$\cdots$, $v_{f(e_{1})+2}$,$\cdots$, and $e_{k}$ is connected to $v_{[\sum\limits_{j=1}^{k-1}f(e_{j})]+2}$, $v_{[\sum\limits_{j=1}^{k-1}f(e_{j})]+3}$,$\cdots$,$v_{[\sum\limits_{j=1}^{k}f(e_{j})]+2}$,$\cdots$, and $e_{i}$ is connected to $v_{[\sum\limits_{j=1}^{i-1}f(e_{j})]+2}$, $v_{[\sum\limits_{j=1}^{i-1}f(e_{j})]+3}$,$\cdots$,$v_{[\sum\limits_{j=1}^{i}f(e_{j})]+2}$, and $e_{q+1}$ is connected to $v_{[\sum\limits_{j=1}^{i}f(e_{j})]+2}$, $v_{[\sum\limits_{j=1}^{i}f(e_{j})]+3}$,$\cdots$,$v_{[\sum\limits_{j=1}^{i}f(e_{j})]+2+f(e_{q+1})}$,$\cdots$, and $e_{q+i'}$ is connected to $v_{[\sum\limits_{j=1}^{i}f(e_{j})+2+\sum\limits_{j=1}^{i'-1}f(e_{q+j})]}$, $v_{[\sum\limits_{j=1}^{i}f(e_{j})+3+\sum\limits_{j=1}^{i'-1}f(e_{q+j})]}$,$\cdots$,$v_{[\sum\limits_{j=1}^{i}f(e_{j})+1+\sum\limits_{j=1}^{i'}f(e_{q+j})]}=v_{m}$ and $v_{1}$.  Moreover, $e$ is connected to $v_{1}$ for all $e\in E\setminus \{e_{1},\cdots,e_{i},e_{q+1},\cdots,e_{q+i'}\}$. It is clear that $\sum\limits_{j=1}^{k-1}f(e_{j})+2>k$ for $k=1,2,\cdots,i$. Compared with complete bipartite graphs, $G$ has only no hypertrees $f_{i}$ $(1=1,2,\cdots,q)$ with $f_{i}(e_{i})=m-1$ and  $f_{i}(e)=0$ for all $e\neq e_{i}$.

We firstly consider the interior polynomial, it is clear that $\overline{\iota} (f_{1})=0$ and $\overline{\iota} (f_{i})=1$ $(i=2,\cdots,q)$ for $K_{m,n}$. Moreover,  $\overline{\iota}(f)=0$ in $G$ and $\overline{\iota} (f)=1$ in $K_{m,n}$ for the hypertree $f=(m-2,1,0,\cdots,0)$. And there are same internal inactivity for all $f'\neq f,f_{i}$ $(i=1,2,\cdots,q)$ in $K_{m,n}$ and $G$. Thus, $I_{G}(x)=1+[(n-1)(m-1)-q]x+\sum\limits_{i=2}^{m-1}\binom{n-1}{i}\binom{m-1}{i}x^{i}$.

For the exterior polynomial, we assume that vertices of $E$ are regarded as hyperdeges. It is clear that $\overline{\epsilon}(f_{i})=n-i$ $(1=1,2,\cdots,q)$ for $K_{m,n}$. Moreover, $\overline{\epsilon}(f)=n-j-1$ in $G$ and $\overline{\epsilon}(f)=n-j$ in $K_{m,n}$ for the hypertree $f$ with $f(e_{i})=m-2$, $f(e_{j})=1$ for some hyperedge $e_{j}<e_{i}$, and $f(e)=0$ for all $e\neq e_{i}, e_{j}$ for any $i=1,2,\cdots,q$. There are same external inactivity for all $f'\neq f,f_{i}$ $(i=1,2,\cdots,q)$ in $K_{m,n}$ and $G$. Thus, by simple calculations, we obtain $X_{G}(y)=\sum\limits_{i=0}^{n-2}\binom{m+i-2}{i}y^{i}+[\binom{m+n-3}{n-1}-q]y^{n-1}$.
\end{proof}
\section*{Acknowledgements}
\noindent
This work is supported by NSFC (No. 12171402) and the Fundamental Research
Funds for the Central Universities  (No. 20720190062).

\section*{References}
\bibliographystyle{model1b-num-names}
\bibliography{<your-bib-database>}

\end{document}